\documentclass[a4,12pt]{amsart}
\usepackage{amssymb,amstext,amsmath,amscd,amsthm,amsfonts,enumerate,latexsym}
\usepackage{color}
\usepackage[dvipdfmx]{graphicx}
\usepackage[all]{xy}
\theoremstyle{plain}
\newtheorem{thm}{Theorem}[section]
\newtheorem*{thm*}{Theorem}
\newtheorem*{cor*}{Corollary}

\newtheorem{prop}[thm]{Proposition}
\newtheorem{lem}[thm]{Lemma}

\newtheorem{claim}{Claim}
\newtheorem*{claim*}{Claim}

\theoremstyle{definition}
\newtheorem{defn}[thm]{Definition}

\newtheorem{conj}[thm]{Conjecture}

\theoremstyle{remark}

\numberwithin{equation}{thm}

\def\Im{\operatorname{Im}}
\def\Ker{\operatorname{Ker}}

\def\Coker{\mathrm{Coker}}

\def\rank{\mathrm{rank}}
\def\a{\mathrm a}

\def\e{\mathrm{e}}
\def\m{\mathfrak m}

\def\q{\mathfrak q}

\newcommand{\rma}{\mathrm{a}}

\newcommand{\rme}{\mathrm{e}}

\newcommand{\rmH}{\mathrm{H}}

\newcommand{\rmK}{\mathrm{K}}

\newcommand{\calC}{\mathcal{C}}

\newcommand{\calR}{\mathcal{R}}

\newcommand{\fka}{\mathfrak{a}}

\newcommand{\fkm}{\mathfrak{m}}

\newcommand{\fkq}{\mathfrak{q}}

\newcommand{\mapright}[1]{%
\smash{\mathop{%
\hbox to 1cm{\rightarrowfill}}\limits^{#1}}}

\newcommand{\mapleft}[1]{%
\smash{\mathop{%
\hbox to 1cm{\leftarrowfill}}\limits_{#1}}}

\def\depth{\operatorname{depth}}

\begin{document}
\setlength{\baselineskip}{13.5pt}

\title[Almost Gorenstein graded rings]{When are the Rees algebras of parameter ideals almost Gorenstein graded rings?}

\author{Shiro Goto}
\address{Department of Mathematics, School of Science and Technology, Meiji University, 1-1-1 Higashi-mita, Tama-ku, Kawasaki 214-8571, Japan}
\email{shirogoto@gmail.com}
\author{Rahimi Mehran}
\address{Faculty of Mathematics and Computer Science, Kharazmi University, No. 43 Shahid Mofatteh Ave., Tehran, Iran}
\email{mehran${}_{-}$rahimy@yahoo.com}
\author{Naoki Taniguchi}
\address{Department of Mathematics, School of Science and Technology, Meiji University, 1-1-1 Higashi-mita, Tama-ku, Kawasaki 214-8571, Japan}\email{taniguti@math.meiji.ac.jp}
\author{Hoang Le Truong}
\address{Institute of Mathematics, Vietnam Academy of Science and Technology, 18 Hoang Quoc Viet Road, 10307 Hanoi, Vietnam}
\email{hltruong@math.ac.vn}

\thanks{The first author was partially supported by JSPS Grant-in-Aid for Scientific Research (C) 25400051. The third author was partially supported by Grant-in-Aid for JSPS Fellows 26-126 and by JSPS Research Fellow.
The fourth author was partially supported by Vietnam National Foundation for Science and Technology Development (NAFOSTED) under grant number 101.04-2014.15 and VAST.DLT 01/16-17}

\maketitle

\begin{abstract}
Let $A$ be a Cohen-Macaulay local ring with $\dim A = d\ge 3$, possessing the canonical module $\rmK_A$. Let  $a_1, a_2, \ldots, a_r$ ~$(3 \le r \le d)$ be a subsystem of parameters of $A$ and set $Q= (a_1, a_2, \ldots, a_r)$. It is shown that if the Rees algebra $\calR(Q)$ of $Q$ is an almost Gorenstein graded ring, then $A$ is a regular local ring and $a_1, a_2, \ldots, a_r$ is a part of a regular system of parameters of $A$.
\end{abstract}


\section{Introduction}\label{introduction}
The purpose of this note is to study the question of when the Rees algebras of ideals generated by subsystems of parameters in a Cohen-Macaulay local ring are almost Gorenstein graded rings.

For the last sixty years commutative algebra has been concentrated mostly in the study of Cohen-Macaulay rings/modules and experiences in our researches show that Gorenstein rings are rather isolated in the class of Cohen-Macaulay rings. Gorenstein local rings are, of course, defined by the finiteness of self-injective dimension. However there is a substantial gap between the conditions of the finiteness of self-injective dimension and the infiniteness of it. The notion of almost Gorenstein ring is an attempt to fill this gap and a desire to find a new class of Cohen-Macaulay rings which might be non-Gorenstein but still good, say the next to Gorenstein rings.

The notion of almost Gorenstein local ring in our sense dates back to the paper \cite{BF} of V. Barucci and R. Fr\"oberg in 1997, where they introduced the notion to one-dimensional analytically unramified local rings and developed a very beautiful theory of almost symmetric numerical semigroups. Because their definition is not flexible enough for the analysis of analytically ramified local rings, S. Goto, N. Matsuoka, and T. T. Phuong \cite{GMP} relaxed in 2013 the restriction and gave the definition of almost Gorenstein local rings for arbitrary but still one-dimensional Cohen-Macaulay local rings, using the first Hilbert coefficients of canonical ideals. In \cite{GMP} they constructed numerous examples of almost Gorenstein local rings which are analytically ramified, extending several results of \cite{BF}. However it might be the most striking achievement of \cite{GMP} that the paper prepared for the higher dimensional definition and opened the door led to the theory of higher dimension. In fact in 2015 S. Goto, R. Takahashi, and N. Taniguchi \cite{GTT}  gave the definition of almost Gorenstein local/graded rings of higher dimension and started the theory.

Let us recall their definition.

\begin{defn}[The local case]\label{1.1}
Let $(A,\fkm)$ be a Cohen-Macaulay local ring of dimension $d$, possessing the canonical module $\rmK_A$. Then we say that $A$ is an almost Gorenstein local ring, if there exists an exact sequence
$$
0 \to A \to \rmK_A \to C \to 0
$$
of $A$-modules such that either $C = (0)$ or $C \ne (0)$ and $\mu_A(C) = \e^0_\fkm(C)$, where  $\mu_A(C)$ denotes the number of elements in a minimal system of generators of $C$ and $$\e^0_\fkm(C) = \lim_{n\to \infty}(d-1)!{\cdot}\frac{\ell_A(C/\fkm^{n+1}C)}{n^{d-1}}$$ denotes the multiplicity of $C$ with respect to the maximal ideal $\fkm$ (here $\ell_A(*)$ stands for the length).
\end{defn}

Let us explain a little more about Definition \ref{1.1}. Let $(A,\fkm)$ be a Cohen-Macaulay local ring of dimension $d$ and assume that $A$ possesses the canonical module $\rmK_A$. The condition of Definition \ref{1.1} requires that  $A$ is embedded into   $\rmK_A$ and even though $A \ne \rmK_A$, the difference $C = \rmK_A/A$ between $\rmK_A$ and $A$ is an Ulrich $A$-module (cf. \cite{BHU}) and behaves well. Here we notice that for every exact sequence$$
0 \to A \to \rmK_A \to C \to 0
$$
of $A$-modules, $C$ is a Cohen-Macaulay $A$-module of dimension $d-1$, provided $C \ne (0)$ (\cite[Lemma 3.1 (2)]{GTT}).

\begin{defn}[The graded case]\label{1.2}
Let $R = \sum_{n \ge 0}R_n$ be a Cohen-Macaulay graded ring such that $A = R_0$ is a local ring. Assume that $A$ is a homomorphic image of a Gorenstein local ring and let $\rmK_R$ denote the graded canonical module of $R$. We set $d = \dim R$ and $a = \a(R)$ the $a$-invariant of $R$. Then we say that $R$ is an almost Gorenstein graded ring, if there exists an exact sequence 
$$
0 \to R \to \rmK_R(-a) \to C \to 0
$$
of graded $R$-modules such that either $C = (0)$ or $C \ne (0)$ and $\mu_R(C) = \e^0_M(C)$, where $M$ denotes the graded maximal ideal of $R$. 
\end{defn}

In Definition \ref{1.2} suppose $C \ne (0)$. Then $C$ is a Cohen-Macaulay graded $R$-module and $\dim_RC  = d -1$.  As 
$
\e_M^0(C) = \lim_{n\to\infty}(d-1)!{\cdot}\frac{\ell_R(C/M^{n+1}C)}{n^{d-1}},
$
we get $\e_{MR_M}^0(C_M)= \e^0_M(C)$, so that $C_M$ is an Ulrich $R_M$-module. Therefore since $\rmK_{R_M}= \left[\rmK_R\right]_M$, $R_M$ is an almost Gorenstein local ring if $R$ is an almost Gorenstein graded ring. The converse is not true in general (\cite[Theorems 2.7, 2.8]{GMTY}, \cite[Example 8.8]{GTT}).

The present research has been motivated by \cite{GMTY}  and comes from a natural question of when the Rees algebras of ideals and modules are almost Gorenstein graded rings. Here we notice that the condition of the almost Gorenstein property in Rees algebras is a rather strong restriction. For example, let $(A,\fkm)$ be a Gorenstein local ring with $d = \dim A \ge 3$ and let $Q$ be a parameter ideal of $A$. Then the Rees algebra $\calR (Q)$ of $Q$ is an almost Gorenstein graded ring if and only if $Q = \fkm$ (\cite[Theorem 8.3]{GTT}). Therefore when this is the case, $A$ is a regular local ring. This result was more closely analyzed in \cite{GMTY} and the authors showed among other results that when $Q$ is an ideal of generated by a subsystem $a_1, a_2, \ldots, a_r$ of parameter with $3 \le r \le d= \dim A$, the Rees algebra $\calR(Q)$ is an almost Gorenstein graded ring if and only if $A$ is a regular local ring and $a_1, a_2, \ldots, a_r$ is a part of a regular system of parameters of $A$, while $\calR(Q)_M$ is an almost Gorenstein local ring if and only if $A$ is  a regular local ring, where $M$ denotes the graded maximal ideal in $\calR(Q)$. We should note here that for all these results the authors assume the base local ring $A$ is a Gorenstein ring. It seems natural to ask if this assumption is really necessary. Because the almost Gorenstein property in Rees algebras is a  strong restriction, it might be enough just to assume that $A$ is a Cohen-Macaulay local ring which is a homomorphic image of a Gorenstein local ring.

The present paper answers this question affirmatively and our result is stated as follows.

\begin{thm}\label{1.3} Let $A$ be a Cohen-Macaulay local ring with $\dim A = d\ge 3$ and assume that $A$ is a homomorphic image of a Gorenstein local ring. Let  $a_1, a_2, \ldots, a_r$ ~$(3 \le r \le d)$ be a subsystem of parameters of $A$ and set $Q= (a_1, a_2, \ldots, a_r)$. Then the following conditions are equivalent. 
\begin{enumerate}[$(1)$]
\item The Rees algebra $\calR(Q)$ of $Q$ is an almost Gorenstein graded ring.
\item $A$ is a regular local ring and $a_1, a_2, \ldots, a_r$ is a part of a regular system of parameters of $A$.
\end{enumerate}
\end{thm}

As is stated above, our contribution in Theorem \ref{1.3} is the implication (1) $\Rightarrow$ (2) under the weaker assumption that $A$ is a Cohen-Macaulay local ring which is  a homomorphic image of a Gorenstein local ring. The implication (2) $\Rightarrow$ (1) is due to \cite[Theorem 2.8]{GMTY}. Our method of proof of Theorem \ref{1.3} is to give a whole proof of the implication (1) $\Rightarrow$ (2) and does not directly deduce the fact that $A$ is a Gorenstein ring once the Rees algebra $\calR(Q)$ is an almost Gorenstein graded ring. Therefore the following conjecture is still open.

\begin{conj}
Let $A$ be a Cohen-Macaulay local ring and assume that $A$ is a homomorphic image of a Gorenstein local ring. Let $I \subsetneq A$ be an ideal of $A$ with $\mathrm{ht}_AI \ge 2$. If the Rees algebra $\calR(I)$ of $I$ is an almost Gorenstein graded ring, then $A$ is a Gorenstein ring.
\end{conj}

To prove Theorem \ref{1.3} we need some preliminaries which we summarize in Section \ref{preliminaries}. We shall prove Theorem \ref{1.3} in Section \ref{proofofmaintheorem}.


\section{Preliminaries}\label{preliminaries}

This section is devoted to preliminaries which we need  to prove Theorem \ref{1.3}.

Let $A$ be an arbitrary commutative ring and let $L$ be an $A$-module. Let $n$ and $\ell$ be positive integers and choose elements $x_1, x_2, \ldots, x_{\ell}, a_1, a_2, \ldots, a_{\ell}$ of $A$. Set $\underline{a} = a_1, a_2, \ldots, a_{\ell}$ and $\underline{x} = x_1, x_2, \ldots, x_{\ell}$. For each $\xi =
\left(\begin{smallmatrix}
f_1\\
f_2\\
\vdots\\
f_\ell
\end{smallmatrix}\right)\in L^{\oplus \ell}$ we set
$\underline{a}\xi =\sum_{i=1}^\ell a_if_i$ and $\underline{x}\xi =\sum_{i=1}^\ell x_if_i$ in $L$ and consider the $A$-linear map
$
\varphi : \left(L^{\oplus \ell}\right)^{\oplus n} \longrightarrow L^{\oplus n}
$ given by the $n \times n\ell$ matrix
{\footnotesize
$$ 
\Bbb A=
\begin{pmatrix}
\underline{a} &  &  &  &   \\
\underline{x} & \underline{a} &    &  &   \\
   & \ddots  & \ddots  \\
   &   &    \underline{x}  &  \underline{a} 
\end{pmatrix},
$$}
that is $\varphi(\left(\begin{smallmatrix}
\xi_1\\
\xi_2\\
\vdots\\
\xi_n
\end{smallmatrix}\right)) =\left(\begin{smallmatrix}
\underline{a}\xi_1\\
\underline{x}\xi_1 + \underline{a}\xi_2\\
\vdots\\
\underline{x}\xi_{n-1}+\underline{a}\xi_n
\end{smallmatrix}\right)$ for each $\left(\begin{smallmatrix}
\xi_1\\
\xi_2\\
\vdots\\
\xi_n
\end{smallmatrix}\right) \in (L^{\oplus \ell})^{\oplus n}$. 
Let $Q=(a_1, a_2, \ldots, a_{\ell})$ and let $\rmH_1(\underline{a}; L)$ denote the first Koszul homology module of $L$ generated by  the sequence $\underline{a}=a_1, a_2, \ldots, a_\ell$.  With this notation  we have the following.

\begin{lem}\label{2.1}
Let $\xi_1, \xi_2, \ldots, \xi_n \in L^{\oplus \ell}$ and suppose that 
$\left(\begin{smallmatrix}
\xi_1 \\
\xi_2\\
\vdots \\
\xi_n
\end{smallmatrix}\right) \in \Ker \varphi$. If $\rmH_1(\underline{a}; L) = (0)$, then $\underline{x}\xi_n \in QL$.
\end{lem}

\begin{proof}
When $n=1$, the assertion is clear. 
Suppose that $n>1$. We consider two Koszul complexes $\rmK_{\bullet}(\underline{a}; L)$ and $\rmK_{\bullet}(\underline{x}; L)$ of $L$ generated by $\underline{a} = a_1, a_2, \ldots, a_{\ell}$ and $\underline{x}= x_1, x_2, \ldots, x_{\ell}$, respectively. More precisely, let $F$ be a finitely generated free $A$-module with $\rank_AF = \ell$ and a free basis $\{T_i\}_{1 \le i \le \ell}$. Let $K =\wedge F$ be the exterior algebra of $F$ and consider two differentiations $\partial^{\underline{a}}$ and $\partial^{\underline{x}}$ on $K$ such that
\begin{center}
$\partial_1^{\underline{a}}(T_i) = a_i$ \   and \   $\partial_1^{\underline{x}}(T_i) = x_i$ 
\end{center}
for all $1 \le i \le \ell$, making $K$ into the Koszul complexes $\rmK_\bullet (\underline{a};A)$ and $\rmK_\bullet(\underline{x};A)$, respectively. For simplicity let us denote  by $\partial^{\underline{a}}$ and $\partial^{\underline{x}}$ also the differentiations of the Koszul complexes $\rmK_\bullet(\underline{a};L) = \rmK_\bullet(\underline{a};A)\otimes_AL$ and $\rmK_\bullet(\underline{x};L) = \rmK_\bullet(\underline{x};A)\otimes_AL$, respectively.

Let us now suppose that 
$\left(\begin{smallmatrix}
\xi_1 \\
\xi_2\\
\vdots \\
\xi_n
\end{smallmatrix}\right) \in \Ker \varphi$. We set $K_2 = \wedge^2F$ and write  $\xi_{\alpha} = \left(\begin{smallmatrix}\xi_{\alpha 1} \\ \xi_{\alpha 2}\\\vdots \\ \xi_{\alpha\ell} \end{smallmatrix}\right)$ for each $1 \le \alpha \le n$. Then since $\partial^{\underline{a}}_1\left(\sum_{i=1}^{\ell}T_i\otimes \xi_{1i}\right)=\underline{a}{\cdot}\xi_1=0$, there exists elements $\rho_1, \rho_0 \in \rmK_2\otimes_AL$ such that 
$$
\sum_{i=1}^{\ell}T_i\otimes\xi_{1i} = \partial^{\underline{a}}_2(\rho_1) + \partial^{\underline{x}}_2(\rho_0)
$$
(take $\rho_0 = 0$ to be the initial data). We then have
\begin{eqnarray*}
0 = \underline{x}\xi_1 + \underline{a}\xi_{2} &=& \partial^{\underline{x}}_1(\sum_{i=1}^{\ell}T_i\otimes \xi_{1i}) + \underline{a}\xi_2 \\
 &=& \partial^{\underline{x}}_1(\partial^{\underline{a}}_2(\rho_1)) + \underline{a}\xi_2 \\
 &=& - \partial^{\underline{a}}_1(\partial^{\underline{x}}_2(\rho_1)) + \partial^{\underline{a}}_1(\sum_{i=1}^{\ell}T_i\otimes \xi_{2i}) \\
 &=& \partial^{\underline{a}}_1\left(\sum_{i=1}^{\ell} T_i\otimes \xi_{2i} - \partial^{\underline{x}}_2(\rho_1)\right).
\end{eqnarray*}
Because $\rmH_1(\underline{a};L) = (0)$, we get 
$\sum_{i=1}^{\ell}T_i\otimes \xi_{2i} - \partial^{\underline{x}}_2(\rho_1) = \partial^{\underline{a}}_2(\rho_2)$
for some $\rho_2 \in \rmK_2\otimes_AL$,
whence
$$
\sum_{i=1}^{\ell}T_i\otimes \xi_{2i}= \partial^{\underline{a}}_2(\rho_2) + \partial^{\underline{x}}_2(\rho_1)
$$
with $\rho_2, \rho_1 \in \rmK_2\otimes_AL$.
Repeat this procedure and we have
$$
\sum_{i=1}^{\ell}T_i\otimes \xi_{ni}= \partial^{\underline{a}}_2(\rho_n) + \partial^{\underline{x}}_2(\rho_{n-1})
$$
for some $\rho_n, \rho_{n-1}\in \rmK_2\otimes_AL$.
Consequently 
$$
\underline{x}\xi_n = \partial_1^{\underline{x}}(\sum_{i=1}^{\ell}T_i\otimes \xi_{ni}) = \partial^{\underline{x}}_1(\partial^{\underline{a}}_2(\rho_n)) \in QL
$$
as claimed.
\end{proof}

We now furthermore assume that $(A, \m)$ is a Noetherian local ring and that $L$ is a non-zero finitely generated $A$-module. We set $D = \Coker \varphi$ and let $\varepsilon:  L^{\oplus n} \overset{\varepsilon}{\longrightarrow} D$ denote the canonical map. Hence we get the exact sequence
$$
\left(L^{\oplus \ell}\right)^{\oplus n} \overset{\varphi}{\longrightarrow} L^{\oplus n} \overset{\varepsilon}{\longrightarrow} D  \longrightarrow 0.
$$

\begin{prop}\label{2.2}
Suppose that $\underline{a}=a_1, a_2, \ldots, a_{\ell}$ forms an $L$-regular sequence. Then 
\begin{enumerate}[$(1)$]
\item $D \ne (0)$, $Q^n D=(0)$, $\dim_A D = \dim_A L - \ell$, and $\depth_A D = \depth_A L - \ell$. 
\item If $L$ is a Cohen-Macaulay $A$-module, then $D$ is a Cohen-Macaulay $A$-module with $\dim_A L - \ell$.
\end{enumerate}
\end{prop}

\begin{proof}
Assertion (2) readily follows from assertion $(1)$. We  prove  assertion (1)  by induction on $n$.
If $n=1$, then $D \cong L/QL$ and we have nothing to prove. Suppose that $n>1$ and assertion $(1)$ holds true for $n-1$. Let $\beta : L \to L^{\oplus n}$ be the homomorphism defined by $\beta(y) = \left(\begin{smallmatrix}
0\\
\vdots\\
0\\
y\\
\end{smallmatrix}\right)$ for each $y \in L$ and consider the composite map
$$
\alpha : L \overset{\beta}{\to} L^{\oplus n} \overset{\varepsilon}{\to} D,~~y \mapsto \overline{
\left(\begin{smallmatrix}
0\\
\vdots\\
0\\
y
\end{smallmatrix}\right)}
$$
where $\overline{z}$ denotes for each $z \in L^{\oplus n}$ the image of $z$ in $D$. 
We set $E=\Im \alpha$ and $\overline{D} = D/E$.
We then have  the exact sequence
$$
\left(L^{\oplus \ell}\right)^{\oplus (n-1)} \overset{\psi}{\longrightarrow} L^{\oplus (n-1)} \longrightarrow \overline{D} \longrightarrow 0
$$
of $A$-modules, where $\psi$ is given by the $(n-1)\times (n-1)\ell $ matrix
{\footnotesize
$$ \Bbb B=
\begin{pmatrix}
\underline{a} &  &  &  &   \\
\underline{x} & \underline{a} &    &  &   \\
   & \ddots  & \ddots  \\
   &   &    \underline{x}  &  \underline{a} 
\end{pmatrix}.
$$}
Therefore, thanks to the exact sequence $0 \to E \to D \to \overline{D} \to 0$,  assertion $(1)$ directly follows from the hypothesis of induction, once we get $E \cong L/QL$. Let $y \in L$ and notice that 
\begin{center}
$y \in \Ker \alpha$ \  if and only if  \ 
$\begin{pmatrix}
0\\
\vdots\\
0\\
y
\end{pmatrix}=
\begin{pmatrix}
\underline{a} &  &  &  &   \\
\underline{x} & \underline{a} &    &  &   \\
   & \ddots  & \ddots  \\
   &   &    \underline{x}  &  \underline{a} 
\end{pmatrix}
\begin{pmatrix}
\xi_1\\
\xi_2\\
\vdots\\
\xi_n
\end{pmatrix}$
\end{center}
for some $\xi_1, \xi_2, \ldots, \xi_n \in L^{\oplus \ell}$. When this is the case,  $\underline{x}\xi_{n-1} \in QL$ by Lemma \ref{2.1}, so that $y = \underline{x}\xi_{n-1} + \underline{a}\xi_n \in QL$. If $y \in QL$, then
\begin{center}
$\begin{pmatrix}
0\\
\vdots\\
0\\
y
\end{pmatrix}=
\begin{pmatrix}
\underline{a} &  &  &  &   \\
\underline{x} & \underline{a} &    &  &   \\
   & \ddots  & \ddots  \\
   &   &    \underline{x}  &  \underline{a} 
\end{pmatrix}
\begin{pmatrix}
0\\
\vdots\\
0\\
\xi
\end{pmatrix}$
\end{center}
for some $\xi \in L^{\oplus \ell}$. Thus $\Ker \alpha =QL$. Hence $E \cong L/QL$ as wanted.
\end{proof}


\section{Proof of Theorem \ref{1.3}}\label{proofofmaintheorem}

The purpose of this section is to prove Theorem \ref{1.3}. Let $(A, \m)$ be a Cohen-Macaulay local ring with $d= \dim A \ge 3$. 
Let $a_1, a_2, \ldots, a_r ~ (r \ge 3)$ be a subsystem of parameters of $A$ and set $Q = (a_1, a_2, \ldots, a_r)$. We denote by
$$
R = \calR(Q) = A[Qt] \subseteq A[t]
$$
the Rees algebra of $Q$ where $t$ stands for an indeterminate over $A$. Hence $R$ is a Cohen-Macaulay ring with $\dim R = d+1$ and $a(R) = -1$.
Let $S=A[X_1, X_2, \ldots, X_r]$ be the  polynomial ring which we consider to be a standard graded $A$-algebra and set $N = \fkm S + S_+$, the graded maximal ideal of $S$. Let $\varphi : S \longrightarrow R$ be the homomorphism of $A$-algebras defined by  $\varphi(X_i) = a_i t$ for each $1 \le i \le r$. We set 
$$
{\Bbb X} =\begin{pmatrix}
X_1 & X_2 & \cdots & X_r \\
a_1 & a_2 & \cdots & a_r 
\end{pmatrix}.
$$  
Then $\Ker \varphi$ is generated by the $2 \times 2$ minors of the matrix $\Bbb X$, that is 
$$
\Ker \varphi = {\rm I}_2
\begin{pmatrix}
X_1 & X_2 & \cdots & X_r \\
a_1 & a_2 & \cdots & a_r 
\end{pmatrix}
$$
which is a perfect ideal of $S$ with grade $r-1$. Therefore a minimal graded $S$-free resolution of $R$ is given by the Eagon-Northcott complex associated to the matrix $\Bbb X$ (\cite{EN}).

For later use let us briefly recall the construction of the Eagon-Northcott complex. Let $F$ be a finitely generated free $S$-module with $\rank_SF = r$ and a free basis $\{T_i\}_{1 \le i \le r}$. We denote by $K = \wedge F$ the exterior algebra of $F$ over $S$. Let $\rmK_{\bullet}(X_1, X_2, \ldots, X_r; S)$ (resp. $\rmK_{\bullet}(a_1, a_2, \ldots, a_r; S)$) be the Koszul complex of $S$ generated by $X_1, X_2, \ldots, X_r$ (resp. $a_1, a_2, \ldots, a_r$) with differentiation $\partial_1$ (resp. $\partial_2$). Let $U = S[Y_1, Y_2]$ be the polynomial ring. We set $C_0 = S$ and $C_q = K_{q+1} \otimes_S U_{q-1}$ for each $1 \le q \le r-1$. Hence $C_q$ is a finitely generated free $S$-module with a free basis
$$
\{ T_{i_1}T_{i_2} \cdots T_{i_{q+1}} \otimes Y_1^{\nu_1}Y_2^{\nu_2} \mid 1 \le i_1 < i_2 < \cdots < i_{q+1} \le r,~\nu_1 + \nu_2 =  q-1 \}.
$$
We regard $C_q$ to be a graded $S$-module such that 
$$\deg(T_{i_1}T_{i_2}\ldots T_{i_{q+1}}\otimes Y_1^{\nu_1}Y_2^{\nu_2}) = \nu_1+1.$$
With this notation the Eagon-Northcott complex associated with $\Bbb X$ is defined to be a complex of graded $S$-modules of the form
$$
\calC_{\bullet} \ \ : \ \  0 \to C_{r-1} \overset{d_{r-1}}{\rightarrow} C_{r-2} \to \cdots \to C_1 \overset{d_{1}}{\rightarrow} C_0 \to 0
$$
where 
$$
d_q(T_{i_1}T_{i_2} \cdots T_{i_{q+1}} \otimes Y_1^{\nu_1}Y_2^{\nu_2}) = \sum_{j=1, 2 ~\text{and}~ \nu_j > 0} \partial_j(T_{i_1}T_{i_2} \cdots T_{i_{q+1}}) \otimes Y_1^{\nu_1}\cdots Y_j^{{\nu_j}-1} \cdots Y_2^{\nu_2}
$$
for $q \ge 2$ and 
$$
d_1(T_{i_1}T_{i_2}\otimes 1) = \det
\begin{pmatrix}
X_{i_1} & X_{i_2} \\
a_{i_1} & a_{i_2} 
\end{pmatrix},
$$
whence $\Im d_1= {\rm I }_2(\Bbb X) \subseteq S$.  Then the complex $C_\bullet$ is a graded minimal $S$-free resolution of $R$, since  ${\rm I}_2(\Bbb X)$ is a perfect ideal of grade $r-1$ and $X_i, a_i \in N = \m S + S_+$ for all $1 \le i \le r$ (\cite{EN}).

We are now especially interested in the presentation matrix $\Bbb M$ of the homomorphism $C_{r-1} \overset{d_{r-1}}{\longrightarrow} C_{r-2}$ with respect to the basis $$\{T_1T_2\cdots T_r\otimes Y_1^iY_2^{r-2-i}\}_{0 \le i \le r-2}$$ and $$\{T_1 \cdots \overset{\vee}{T_j} \cdots T_r \otimes Y_1^k Y_2^{r-3-k}\}_{1 \le j \le r, ~0 \le k \le r-3}$$ 
of $C_{r-1}$ and $C_{r-2}$, respectively. Notice that $\Bbb M$ is an $(r-2)r \times (r-1)$ matrix. Then a direct computation shows
{\footnotesize
$$
{}^t{\Bbb M} = 
\begin{pmatrix}
\underline{a} & 0 &  &  &  &  \\
\underline{X} & \underline{a} &    &  &  &  \\
   &   & \ddots  & \\
   &   &  &  &  \underline{X}  &  \underline{a}\\   
   &   &  &  &  0  &  \underline{X}   
\end{pmatrix}
$$}
where $\underline{a} = a_1, -a_2, \cdots, (-1)^{r+1} a_r$ and \ $\underline{X} = X_1, -X_2, \cdots, (-1)^{r+1} X_r$ and taking the $S(-r)$-dual of $d_{r-1}$ with acounting degrees, we get the homomorphism of graded $S$-modules
\begin{equation*}
\begin{matrix}
S\left(-r + 1\right)^{\oplus r} \\
\oplus \\
\vdots \\
\oplus \\
S\left(-2\right)^{\oplus r} 
\end{matrix}
\overset{{}^t\Bbb M}{\longrightarrow}
\begin{matrix}
S\left(-r + 1\right) \\
\oplus \\
\vdots \\
\oplus \\
S\left(-1\right). 
\end{matrix}
\end{equation*}


Now suppose that $A$ is a homomorphic image of a Gorenstein local ring and let $\rmK_A$ be the canonical module of $A$. We set $L=S \otimes_A \rmK_A$. Then $\rmK_S = L(-r)$, so that taking the $\rmK_S$-dual of the Eagon-Northcott resolution, we have the following presentation of the graded canonical module $\rmK_R$ of $R$.

\begin{prop}\label{3.1}
\begin{equation*}
\begin{matrix}
L\left(-r + 1\right)^{\oplus r} \\
\oplus \\
\vdots \\
\oplus \\
L\left(-2\right)^{\oplus r} 
\end{matrix}
\overset{{}^t\Bbb M}{\longrightarrow}
\begin{matrix}
L\left(-r + 1\right) \\
\oplus \\
\vdots \\
\oplus \\
L\left(-1\right) 
\end{matrix}
\overset{\varepsilon}{\longrightarrow}
\rmK_R \longrightarrow 0. \quad \quad 
\end{equation*}
\end{prop}


We are now ready to prove Theorem \ref{1.1}.

\begin{proof}[Proof of $(1) \Rightarrow (2)$ in Theorem \ref{1.1}] 
Enlarging the residue class field of $A$ if necessary, we may assume the field $A/\fkm$ is infinite (\cite[Theorem 3.9]{GTT}). We choose an exact sequence
$$
0 \to R \overset{\psi}{\longrightarrow} \rmK_R(1) \to C \to 0
$$
of graded $R$-modules such that either $C=(0)$ or $C\ne (0)$ and $C_M$ is an Ulrich $R_M$-module (remember that $\rma (R) = -1$). We actually have  $C \ne (0)$, since $\mu_R(\rmK_R) =(r-1)\cdot \mu_A(\rmK_A) \ge 2$ by  Proposition \ref{3.1}. We set
$$
D=C/RC_0 \cong \left(\rmK_R/R{\cdot}[\rmK_R]_1\right)(1).
$$
Hence the sequence
$$
0 \to RC_0 \to C \to D \to 0 \ \ \ \ \ \ \mathrm{(}E\mathrm{)}
$$
of graded $R$-modules is exact and because $\psi(1) \in [\rmK_A]_1$, by Proposition \ref{3.1} we readily get  the presentation 
\begin{equation*}
\begin{matrix}
L\left(-r + 1\right)^{\oplus r} \\
\oplus \\
\vdots \\
\oplus \\
L\left(-2\right)^{\oplus r} 
\end{matrix}
\overset{\Bbb A}{\longrightarrow}
\begin{matrix}
L\left(-r + 1\right) \\
\oplus \\
\vdots \\
\oplus \\
L\left(-2\right) 
\end{matrix}
\overset{\varepsilon}{\longrightarrow}
D(-1)  \longrightarrow 0   
\end{equation*}
of $D(-1)=\rmK_R/R{\cdot}[\rmK_R]_1$ as a graded $S$-module, where $\Bbb A$ is an $(r-2)\times (r-2)r$ matrix of the form 
{\footnotesize
$$
\Bbb A = 
\begin{pmatrix}
\underline{a} &  &  &  &   \\
\underline{X} & \underline{a} &    &  &   \\
   & \ddots  & \ddots  \\
   &   &    \underline{X}  &  \underline{a} 
\end{pmatrix}.
$$}
Therefore $D$ is a Cohen-Macaulay $S$-module  with  $\dim_S D =d$ by Proposition \ref{2.2}. Setting $\fka = (X_1, X_2, \ldots, X_r)S$, by the above presentation of $D(-1)$  we get isomorphisms

{\small
\begin{eqnarray*}
(D/\fka D)(-1) \cong 
\begin{matrix}
\left(L/(Q+\fka)L\right)\left(-r + 1\right) \\
\oplus \\
\vdots \\
\oplus \\
\left(L/(Q+\fka)L\right)\left(-2\right) 
\end{matrix}
\cong
\begin{matrix}
\left(\rmK_A/ Q \rmK_A\right)\left(-r + 1\right) \\
\oplus \\
\vdots \\
\oplus \\
\left(\rmK_A/ Q \rmK_A\right)\left(-2\right) 
\end{matrix} \quad \quad\quad(*)
\end{eqnarray*}}
also.

\begin{claim}\label{claim}
$\rme^0_N(D)= \mu_R(D)$.
\end{claim}

\begin{proof}[Proof of Claim \ref{claim}]
Since both $C$ and $D$ are Cohen-Macaulay $S$-modules with $\dim_SC = \dim_SD = d$, $RC_0$ is also a Cohen-Macaulay $S$-module of dimension $d$ if $ RC_0 \ne (0)$. Therefore $\mu_S(RC_0) \le \e_N^0(RC_0)$ and $\mu_S(D) \le \e_N^0(D)$, while by exact sequence ($E$) we get
 \begin{eqnarray*}
\rme_N^0(C)  &=& \rme^0_N(D) + \rme^0_N(R C_0) \quad \text{and} \\
\mu_R(C) &\le& \mu_R(D) + \mu_R(R C_0).
\end{eqnarray*}
Hence $\rme^0_N(D)= \mu_R(D)$ because $\rme^0_N(C) = \mu_R(C)$. 
\end{proof}

We have $Q^{r-2}D=(0)$ by Proposition  \ref{2.2}. Hence $D$ is a finitely generated graded $\overline{S}$-module, where
$$
\overline{S}=S/Q^{r-2}S = (A/Q^{r-2})[X_1, X_2, \ldots, X_r].
$$
We choose elements $a_{r+1}, a_{r+2}, \ldots, a_d \in \m$ so that the images of $a_{r+1}, a_{r+2}, \ldots, a_d$ in $A/Q^{r-2}$ generate a reduction of the maximal ideal $\m/Q^{r-2}$ of $A/Q^{r-2}$. Then because $$[(a_{r+1}, a_{r+2}, \ldots, a_d) + (X_1, X_2, \ldots,X_r)]\overline{S}$$ is a reduction of $N\overline{S}$ and $D$ is a Cohen-Macaulay $\overline{S}$-module with $\dim_{\overline{S}}D=d$, we get 
\begin{eqnarray*}
\rme_N^0(D) &=& \ell_A(D/[(a_{r+1}, a_{r+2}, \ldots, a_d) + (X_1, X_2, \ldots,X_r)]D)\\
                         &=& (r-2){\cdot}\ell_A(\rmK_A/(a_1, a_2, \ldots, a_d)\rmK_A) \\
                         &=& (r-2){\cdot}\ell_A(A/(a_1, a_2, \ldots, a_d))
\end{eqnarray*}
where the second equality follows from isomorphisms $(*)$, while $$ \mu_S(D) = (r-2) {\cdot}\mu_A(\rmK_A)$$ by isomorphisms $(*)$. Therefore because $$\mu_A(\rmK_A) = \ell_A([\fkq:\fkm]/\fkq)$$ where $\q = (a_1, a_2, \ldots, a_d)$ (\cite[Satz 6.10]{HK}) and because $r >2$, Claim \ref{claim} guarantees that 
$$
\ell_A(A/\fkq)=\ell_A([\fkq : \m]/ \fkq ).
$$
Consequently $A=\fkq : \m$, whence $\m=\fkq$. Thus  $A$ is a regular local ring and $a_1, a_2, \ldots, a_r$ is a part of a regular system of parameters of $A$.
\end{proof}




\end{document}